\documentclass[12pt]{amsart}
\usepackage{amssymb,latexsym}
\usepackage{enumerate}

\makeatletter
\@namedef{subjclassname@2010}{%
  \textup{2010} Mathematics Subject Classification}
\makeatother

\newtheorem{thm}{Theorem}[section]
\newtheorem{cor}[thm]{Corollary}
\newtheorem{lem}[thm]{Lemma}

\newtheorem{prop}[thm]{Proposition}

\theoremstyle{definition}
\newtheorem{defin}[thm]{Definition}
\newtheorem{rem}[thm]{Remark}

\numberwithin{equation}{section}

\begin{document}

\baselineskip=17pt


\title{On $p$-adic properties of Siegel modular forms}

\author[S. B\"{o}cherer]{Siegfried B\"{o}cherer}
\address{Kunzenhof 4B\\ Freiburg 79117\\
Germany}
\email{boech@rumms.uni-mannheim.de}

\author[S. Nagaoka]{Shoyu Nagaoka}
\address{Department of Mathematics\\Kinki University\\
577-8502 Osaka, Japan}
\email{nagaoka@math.kindai.ac.jp}

\date{}

\begin{abstract}
We show that Siegel modular forms of level $\Gamma_0(p^m)$ are 
$p$-adic modular forms. Moreover we show that derivatives of
such Siegel modular forms are $p$-adic. Parts of our results
are also valid for vector-valued modular forms. 
In our approach to $p$-adic Siegel modular forms we follow Serre \cite{Ser1}
closely; his proofs however do not generalize to the Siegel case
or need some modifications.
\end{abstract}



\maketitle

\section{Introduction}
\label{intro}
Starting with Swinnerton-Dyer \cite{SW} and Serre \cite{Ser1}, 
the mod $p$ properties
of elliptic modular forms and also their $p$-adic properties have 
been deeply studied.
Some aspects of this theory were later generalized to other
types of modular forms like Jacobi forms \cite{SO} and also Siegel modular forms
\cite{ICHI1}.
In our previous works we constructed Siegel modular
form congruent $1$ mod $p$; we did this for level one \cite{BN1} and 
also for 
level $p$ with additional good $p$-adic behavior in the other cusps
\cite{BN2}. \\
In the present paper we
are concerned with generalizing some of Serre's results to the case
of Siegel modular forms. 
In the first part we show that Siegel modular forms for congruence subgroups
$\Gamma^n_0(p^m)$ are always $p$-adic modular forms. For $m=1$ we cannot 
follow Serre directly \cite{Ser1} because certain modular forms of level $p$,
congruent 1 mod $p$ and with divisibility by $p$ in the other cusps 
are not available (there are $n+1$ cusps to be considered!).
The generalization to $m>1$ then works in the same way as in \cite{Ser2},
with some delicate new problem concerning the vector-valued case.\\  
The second part of this paper is concerned with derivatives of modular forms as
$p$-adic modular forms. 
In \cite{BN1} we generalized the $\varTheta$-operator, defined
on elliptic modular forms by $\sum a_nq^n\longmapsto \sum na_nq^n$
to Siegel modular forms. We showed that the algebra of Siegel modular forms
mod $p$ is stable under $\varTheta$. We generalize the $\varTheta$-operator
to a wide class of differential operators appearing in certain Rankin-Cohen
brackets  and show that they define $p$-adic modular forms; we also correct
a mistake in the proof presented in \cite{BN1}.\\ 
Most of our results are also valid for modular forms of real nebentypus.
Sometimes we just mention this
generalization without going into details. 
At some points our methods give results which are weaker for
vector-valued modular forms than for scalar-valued ones. The reason is
that we cannot use the $p$-th power of a modular form in the same way as
for the scalar-valued case; to take the $p$-th symmetric power is a 
good substitute, but it changes the representation space.
A more detailed treatment of the 
vector-valued case will be given in a subsequent work \cite{BN3}.\\ 
Finally we mention that our paper is 
not concerned with the intrinsic theory of Siegel modular 
forms over ${\mathbb F}_p$ (as created by Katz \cite{KA} in
degree $n=1$ and in general by Faltings and Chai \cite{FACH}).
We only deal with mod $p$ reductions of characteristic zero modular forms.
For an approach to $p$-adic Siegel modular forms based on the arithmetic theory
of Faltings-Chai we refer to \cite{ICHI2}.

\section{Preliminary}
\label{sec:2}
\subsection{Siegel modular forms}
\label{sec:2.1}
Let $\mathbb{H}_n$ denote the Siegel upper half space of degree
$n$. The real symplectic group $Sp_n(\mathbb{R})$ acts on
$\mathbb{H}_n$ in usual manner:
\[
Z\longmapsto M\langle Z\rangle :=(AZ+B)(CZ+D)^{-1}
\]
$(Z\in\mathbb{H}_n,M=\binom{A\,B}{C\,D}\in Sp_n(\mathbb{R}))$.\\
Let $(\rho,V_{\rho})$ be a finite dimensional polynomial representation
of $GL_n(\mathbb{C})$.
For any $V_{\rho}$-valued function $F(Z)$ on $\mathbb{H}_n$
and any element $M=\binom{A\,B}{C\,D}\in Sp_n(\mathbb{R})$,
we write
\[
(F\mid_{\rho,k}M)(Z)=\text{det}(CZ+D)^{-k}\rho (CZ+D)^{-1}F(M\langle Z\rangle ).
\]
Let $\Gamma\subset \Gamma^n:=Sp_n(\mathbb{Z})$ be a congruence subgroup and
$v$ a character of $\Gamma$.
\begin{defin}
\label{def:2.1}
\; A $V_\rho$-valued holomorphic function $F$
on $\mathbb{H}_n$ is called a ($V_\rho$-valued) Siegel modular
form of type $\rho\otimes\text{det}^k$ on $\Gamma$ with character $v$ if
\[
(F\mid_{\rho,k}M)(Z)=v(M)F(Z)\quad \text{for}\;\text{all}\;M\in\Gamma
\]
(and $F(Z)$ is bounded at the cusps for $n=1$). 
\end{defin}

We denote by $M_n^k(\Gamma,\rho,v)$
the space of such modular forms. In the case where $\rho$
is the one-dimensional trivial representation, 
then we are in the scalar-valued case and
we write $F\mid_k=F\mid_{\rho,k}$, and 
$M_n^k(\Gamma,v)=M_n^k(\Gamma,\rho,v)$ simply.\\
We will be mainly concerned with the Siegel modular group
$\Gamma^n$ and congruence subgroup
\[
\Gamma_0^n(N):=\left\{\binom{A\,B}{C\,D}\in Sp_n(\mathbb{Z})\mid 
C \equiv O \pmod{N}\right\}.
\]
Moreover, we assume that $v$ comes from a Dirichlet character $\chi$
mod $N$ which as usual acts on the determinant of the right lower
block. If $\chi$ is trivial, we write 
$M_n^k(\Gamma,\rho)$
for  simplicity.\\
If a Siegel modular form $F(Z)$ is periodic with respect to the
lattice $Sym_n(\mathbb{Z})$, then $F(Z)$ admits a Fourier
expansion of the form
\[
F(Z)=\sum_{0\leq T\in\Lambda_n}a_F(T)
\text{exp}(2\pi\sqrt{-1}\text{tr}(TZ)),\quad a_F(T)\in V_\rho,
\]
where
\[
\Lambda_n:=\{T=(t_{ij})\in Sym_n(\mathbb{Q})\mid t_{ii},2t_{ij}
\in\mathbb{Z}\}.
\]
\subsection{$p$-adic modular forms}
\label{sec:2.2}
By fixing a basis of $V_\rho$, we may view 
$\rho$ as a matrix-valued representation ( $V_\rho=\mathbb{C}^M$ for some $M$).
Taking $q_{ij}:=\text{exp}(2\pi\sqrt{-1}z_{ij})$ with $Z=(z_{ij})
\in\mathbb{H}_n$, we write
\[
\boldsymbol{q}^T:=\text{exp}(2\pi\sqrt{-1}\text{tr}(TZ))=
\prod_{i<j}q_{ij}^{2t_{ij}}\prod_{i=1}^nq_{ii}^{t_{ii}}.
\]
Using this notation, we have the generalized $q$-expansion:
\begin{align*}
F=\sum_{0\leq T\in\Lambda_n}a_F(T)\boldsymbol{q}^T
&=\sum\Big(a_F(T)\prod_{i<j}q_{ij}^{2t_{ij}}\Big)
\prod_{i=1}^nq_{ii}^{t_{ii}}\\
&\in\mathbb{C}^M[q_{ij}^{-1},q_{ij}][\![q_{11},\ldots,q_{nn}]\!],
\end{align*}
($a_F(T)=(a_F(T)^{(j)})\in \mathbb{C}^M$).
\\
For any subring $R$ of $\mathbb{C}$, we shall denote by
$M_n^k(\Gamma,\rho,\chi)(R)$ the $R$-module consisting of those
$F$ in $M_n^k(\Gamma,\rho,\chi)$ for which $a_F(T)$ is in $R^M$ for every
$T\in\Lambda_n$. From this, any element $F$ in $M_n^k(\Gamma,\rho,\chi)(R)$
may be regarded as an element of the space of formal power
series
$R^M[q_{ij}^{-1},q_{ij}]$ $[\![q_{11},\ldots,q_{nn}]\!]$.\\
For a prime number $p$, we denote by $\nu_p$ is the normalized
additive valuation on $\mathbb{Q}$ (i.e. $\nu_p(p)=1$). We tacitly
extend $\nu_p$ to appropriate field extensions $\boldsymbol{K}$ of
$\mathbb{Q}$ if necessary. \\
For a Siegel modular form 
$F=\sum a_F(T)\boldsymbol{q}^T\in M_n^k(\Gamma,\rho,\chi)(\boldsymbol{K})$,
we define $\nu_p(F)$ by
\[
\nu_p(F)=\mathop{\text{inf}}_{T\in\Lambda_n}\nu_p(a_F(T)),
\]
where $\nu_p(a_F(T))=\text{min}_{1\leq j\leq M}(\nu_p(a_F(T)^{(j)})$.
\begin{rem}
\label{rem:2.2} In the definition above we do not exclude the 
possibility that $\nu_p(F)$
becomes $-\infty$ in the case of arbitrary $\rho$. 
We do not know a (published) statement about
boundedness of denominators for general $\rho$ (see however Remark \ref{rem:3.7}
for a possible proof and also the preprint \cite{ICHI2} ). We say that
{\it $\nu_p$-boundedness} holds for $M^k_n(\Gamma,\rho,\chi)$,
if $\nu_p(f\mid_{k,\rho} \omega)$ is finite for all 
$f\in M^k_n(\Gamma,\rho,\chi)$ and all $\omega\in \Gamma^n$.  
\end{rem} 
For two Siegel modular forms
$F=\sum a_F(T)\boldsymbol{q}^T\in M_n^k(\Gamma,\rho,v)(\boldsymbol{K})$,
$G=\sum a_G(T)\boldsymbol{q}^T\in M_n^l(\Gamma,\rho,v)(\boldsymbol{K})$,
with $\nu_p(F)>-\infty$ and $\nu_p(G)>-\infty$ we write
\[
F \equiv G \pmod{p^m}
\]
if $\nu_p(a_F(T)-a_G(T))\geq m+\nu_p(F)$ for every $T\in\Lambda_n$.

\begin{defin}
\label{def:2.3}
\;A formal power series
\[
F=\sum a_F(T)\boldsymbol{q}^T\in \mathbb{Q}_p^M[q_{ij}^{-1},q_{ij}]
[\![q_{11},\ldots,q_{nn}]\!]
\]
is called a ($vector$-$valued$) $p$-$adic$ $Siegel$ 
$modular$ $form$ (in the sense
of Serre) if there exists a sequence of modular forms $\{F_m \}$
satisfying
\[
F_m=\sum a_{F_m}(T)\boldsymbol{q}^T\in M_n^{k_m}(\Gamma^n,\rho)(\mathbb{Q})
\quad \text{and}\quad \lim_{m\to\infty}F_m=F,
\]
where $\lim_{m\to\infty}F_m=F$ means that
\[
\mathop{\text{inf}}_{T\in\Lambda_n}(\nu_p(a_{F_m}(T)-a_F(T)))\;\longrightarrow+\infty
\quad (m\to \infty).
\]
\end{defin}
This definition also makes sense if we replace $\mathbb Q$ and $\mathbb{Q}_p$ by 
suitable extension fields.

\subsection{The Hecke operator $U(p)$}
\label{sec:2.3}
Let $F$ be a Siegel modular form in \\
$M_n^k(\Gamma_0^n(N),\rho,\chi)$ with
the Fourier expansion $F=\sum a_F(T)\boldsymbol{q}^T$. The action of
$U(p)$ on $F$ is defined by
\[
F\mid U(p)=\sum a_F(p\,T)\boldsymbol{q}^T.
\]
It is known that $U(p)$ maps the space $M_n^k(\Gamma_0^n(N),\rho,\chi)$
into itself (if $p\!\mid\!N$) and maps it into 
$M_n^k(\Gamma_0^n(\frac{N}{p}),\rho,\chi)$ if $p^2\!\mid\!N$ and $\chi$ is
defined modulo $\frac{N}{p}$. We recall the following result from \cite{BoU}:
\begin{thm}
\label{th:2.4}
The operator $U(p)$ is
bijective for $p\!\mid\!\mid\!N$.
\end{thm}
\subsection{Coset representatives of $\Gamma_0^n(p)\backslash \Gamma^n$}
\label{sec:2.4}
In this subsection, we describe a system of representatives for 
$\Gamma_0^n(p)\backslash \Gamma^n$, to be used later on  to calculate 
a trace operator on modular forms.
For the finite field $\mathbb{F}_p$  let $P=\{\binom{A\,B}{0\;D}\}$
$\subset Sp_n(\mathbb{F}_p)$ be the Siegel parabolic subgroup.
For $0\leq j\leq n$ we define ``partial involutions''
\[
\omega_j=\omega_{j}(p)=
\left(\begin{matrix}
1_{n-j} & 0   & 0_{n-j} & 0\\
0       & 0_j & 0       & -1_j\\
0_{n-j} & 0   & 1_{n-j} & 0\\
0       & 1_j & 0       & 0_j
\end{matrix}\right).
\]
Then we have a Bruhat decomposition
\[
Sp_n(\mathbb{F}_p)=\mathop{\amalg}_{j=0}^n P\omega_jP,
\]
where the double coset $P\omega_jP$ consists of the set of
elements $\binom{A\,B}{C\,D}\in Sp_n(\mathbb{F}_p)$ with
$\text{rank}(C)=j$. Using the Levi decomposition $P=MN$
with Levi factor
\[
M=\left\{ m(A)=\binom{A\;\;\;\;\;\; 0\;\;}{\;\;0\;\;\;\;(A^{-1})^t}\mid
A\in GL_n(\mathbb{F}_p) \right\}
\]
and unipotent radical
\[
N=\left\{
n(B)=\binom{1\;\;B}{0\;\; 1}\mid B\in Sym_n(\mathbb{F}_p)
\right\},
\]
we easily see that
\[
(*)\quad \{\;
\omega_j\,n(B_j)\,m(A)\mid B_j\in Sym_j(\mathbb{F}_p),
A\in P_{n,j}(\mathbb{F}_p)\backslash GL_n(\mathbb{F}_p)
\}
\]
is a complete set of right coset representatives for
$P\backslash P\omega_jP$. Here $M_j(\mathbb{F}_p)$
is embedded into $M_n(\mathbb{F}_p)$ by
$B_j\longmapsto \binom{0\;\;\; 0}{\,\,0\,\;\, B_j}$ and
$P_{n,j}=\{ M\in GL_n\mid M=\binom{\quad *\;\;\;
\;\; *}{0_{j,n-j}\;\;*}\}$
is a standard maximal parabolic subgroup of $GL_n$.

Using strong approximation we obtain 
\begin{prop}
\label{prop:2.5} By lifting the coset
representatives $(*)$ for $0\leq j\leq n$ to elements of $\Gamma^n$, 
we get  a complete set of right coset representatives for
$\Gamma_0^n(p)\backslash \Gamma^n$. 
(We identify the lifts with their image modulo $p$.)
\end{prop}

For $F\in M_n^k(\Gamma_0^n(p),\rho)$, we define the trace
of $F$ as
\[
\text{tr}(F):=\sum_{M\in \Gamma_0^n(p)\backslash \Gamma^n}
F\mid_{\rho,k}M.
\] 
The trace clearly projects
$M_n^k(\Gamma_0^n(p),\rho)$ onto $M_n^k(\Gamma^n,\rho)$.
Using the coset representatives above, we give an explicit
description of the trace operator.\\
Noting that the action of the $n(B_j)$ comes down to an exponential
sum as a factor (equal either to zero or to $p^{\frac{j(j+1)}{2}}$), we obtain
\[
\text{tr}(F)=F+\sum_{j=1}^{n-1}p^{\frac{j(j+1)}{2}}F\mid_{\rho,k}
\omega_j\mid \widetilde{U}_j(p)+p^{\frac{n(n+1)}{2}}F\mid_{\rho,k}\omega_n\mid
\widetilde{U}_n(p).
\]
The action of $\widetilde{U}_{j}(p)$ is defined as follows.\\

For general $j$, if
\[
F\mid_{\rho,k}\omega_j
=\sum_{T\in\Lambda_n}b_j(T)\text{exp}(2\pi\sqrt{-1}
\text{tr}\big(\tfrac{1}{p}TZ\big)),
\]
then $F\mid_{\rho,k}\omega_j\mid\widetilde{U}_j(p)$ is defined as
\[
F\mid_{\rho,k}\omega_j\mid\widetilde{U}_j(p)=
\sum_{T\in\Lambda_n}\Big(\sum_{\widetilde{T}}b_j(\widetilde{T})\Big)
\text{exp}(2\pi\sqrt{-1}\text{tr}(TZ)),
\]
where $\widetilde{T}$ runs over the finite subset of $\Lambda_n$ determined
by $T$. We do not need the explicit shape of the $\widetilde{T}$ in the sequel,
except for the case $j=n$, where the description is much simpler:
\[
F\mid_{\rho,k}\omega_n\mid\widetilde{U}_n(p)=\sum_{T\in\Lambda_n}b_n(p\cdot T)
\text{exp}(2\pi\sqrt{-1}\text{tr}(TZ)).
\]

\section{Modular forms of level $p^m$}
\label{sec:3}

From now on $p$ will always be an odd prime.

\subsection{Modular forms of level $p$ are $p$-adic}
\label{sec:3.1}
To generalize Serre's result about modular forms for $\Gamma_0^n(p)$
being $p$-adic modular forms we cannot follow his strategy directly.
The problem is the (non-) 
existence of a modular form with the necessary properties 
$F\equiv 1 \pmod{p}$ and 
$F\mid\omega_j\equiv 0 \pmod{p} $ for all $j>0$ .
The best result towards the existence of such $F$ is (to the best of our 
knowledge) our work \cite{BN2} where we construct $F$ with
$F\equiv 1 \pmod p$ and $\nu_p(F\mid \omega_j)\geq -\frac{j(j-1)}{2}+1$. 
This is however not sufficient to apply 
Serre's method directly to Siegel modular forms,
because $\nu_p(F\mid_{\rho,k}\omega_j)$ is not necessarily positive for 
$j\geq 2$.
We need a variant of Serre's approach (interchanging the roles of the cusps):

We use a modular form ${\mathcal K}_{p-1}$ on $\Gamma_0^n(p)$
with Fourier coefficients in ${\mathbb Z}$ satisfying
\begin{eqnarray*}
{\mathcal K}_{p-1}\mid\omega_j &\equiv 0 & \pmod{p} \qquad (0\leq j\leq n-1),\\
{\mathcal K}_{p-1}\mid\omega_n &\equiv 1 & \pmod{p}.
\end{eqnarray*}
The existence of such a modular form is not a problem at all:
We may use 
$${\mathcal K}_{p-1}:= p^n \,\cdot\theta^n_L$$
where 
$\theta^n_L$ is the theta series associated with any $p$-special
lattice $L$ of rank $2p-2$ and determinant $p^2$. Here a $\mathbb{Z}$-lattice
$L$ is called $p$-special if there exists an automorphism $\sigma$  of $L$ 
such that $\sigma$ is of order $p$ and 
acts freely on $L\backslash\{\boldsymbol{0}\}$.
The existence of such lattices (for odd $p$ ) is discussed in \cite{BN1,BN2}.

\begin{prop}
\label{prop:3.1}
Let $p$ be an odd prime with $\nu_p$-boundedness for  
$M_n^k(\Gamma^n_0(p),\rho)$. Let $f$ be an element
of $M_n^k(\Gamma_0^n(p),\rho)$.
Then for any $\alpha\in {\mathbb N}$ there exists $\beta\in {\mathbb N}$
(depending on $\alpha, f $)
and $F\in M_n^{k+{\beta}\cdot (p-1)}(\Gamma^n,\rho)$
such that
$$\nu_p(f-F)\geq \nu_p(f)+\alpha.$$
The dependence of $\beta$ on $\alpha$ will be clarified below.
\end{prop}

\begin{proof}
 As usual, we assume $\nu_p(f)=0$. We use an extension of $\nu_p$
to the field generated by all Fourier coefficients of all the 
$f\mid_{k,\rho}\omega_j$.    
For the moment we consider (for an  arbitrary modular form 
$g\in M_n^k(\Gamma_0^n(p),\rho)$
and arbitrary $\beta=\kappa p^{\gamma}$)

$${\rm Tr}_{\beta}(g):=
p^{-\frac{n(n+1)}{2} }\cdot {\rm tr}(g\cdot {\mathcal K}_{p-1}^{\beta}).$$

The trace decomposes into $n+1$ pieces $Y_j$ which we consider separately:\\
For $0\leq j\leq n$ we have to look at

$$Y_j:=p^{\frac{j(j+1)}{2}-\frac{n(n+1)}{2}} \left(g\mid_{k,\rho }\omega_j\cdot 
({\mathcal K}_{p-1}\mid \omega_j)^{\beta}\right)\mid \widetilde{U}_j(p).$$ 
Then for $j<n$ we have
$$\nu_p(Y_j)\geq \frac{j(j+1)}{2}-\frac{n(n+1)}{2}+
 \nu_p(g\mid_{k,\rho}\omega_j)+\nu_p({\mathcal K}_{p-1}\mid\omega_j)\cdot 
\beta .$$
Clearly this becomes large if $\beta$ is large (note that
$\nu_p({\mathcal K}_{p-1}\mid\omega_j)>0$). 

The contribution for $j=n$ needs a more detailed study:
We write $\left({\mathcal K}_{p-1}\mid\omega_n\right)^{\beta}$ as
$1+p^{\gamma +1}X$ with a Fourier series $X$ with 
integral Fourier coefficients.   
Then
\begin{align*}
\left(g\mid_{k,\rho} \omega_n\cdot 
({\mathcal K}_{p-1}^{\beta}\mid\omega_n)\right)&\mid 
\widetilde{U}_n(p)\\
&=g\mid_{k,\rho}\omega_n\mid\widetilde{U}_n(p)+
p^{\gamma+1} \left(g\mid_{k,\rho}\omega_n\cdot X\right)\mid \widetilde{U}_n(p).
\end{align*}
Now we use that the $U(p)$ operator is invertible
as a Hecke operator for $\Gamma_0^n(p)$(cf. Theorem \ref{th:2.4}).
Therefore we may choose $g$ such that 
$$g\mid_{k,\rho}\omega_n\mid \widetilde{U}_n(p)=f.$$
With this choice of $g$ the 
contribution for $j=n$ to the trace of 
$g\cdot {\mathcal K}_{p-1}^{\beta}$ , which we call $Y_n$ satisfies
$$\nu_p(Y_n-f)\geq \gamma+1+ \nu_p(g\mid_{k,\rho} \omega_n).$$
Summarizing this, we see that
$F:={\rm Tr}_{\beta}(g)$ is congruent to $f$ modulo $p^{\alpha}$, 
if we choose $\gamma$ 
to be large enough.
\end{proof}
\begin{rem}
\label{rem:3.2}
We wrote $\beta=\kappa\cdot p^{\gamma}$ in the proof in order
to emphasize the different roles played by $\beta$ and $\gamma$.
We have to choose $\gamma$ large enough to assure the congruence
for $Y_n$, but to make the other $Y_j$ divisible by a high power of $p$
it is sufficient that $\beta$ becomes large.
\end{rem}
 
To get from the proposition above
a statement about $p$-adic modular forms, we need some rationality and integrality
properties:

\begin{prop}
\label{prop:3.3}
Assume that the polynomial representation $\rho:GL(n,{\mathbb C})\longrightarrow
GL(M,{\mathbb C}) $ is ${\mathbb Q}$-rational (i.e. the polynomials defining $\rho$
have rational coefficients). Then the following properties hold:\\
a) $M^k_n(\Gamma^n_0(p),\rho)= 
M^k_n(\Gamma^n_0(p),\rho)({\mathbb Q})\otimes {\mathbb C}$.\\
b) The $\nu_p$-boundedness holds for  $M^k_n(\Gamma^n_0(p),\rho)$\\
c) If $f\in  M^k_n(\Gamma^n_0(p),\rho)$ has rational Fourier coefficents, the same 
is true for $f\mid \omega_n$ and for $tr(f)$.
\end{prop}
\begin{proof} First we remark that it is enough to prove these statements for large weights
(by multiplying the modular forms in question by a level one modular form $G$ with integral
Fourier coefficients and $G\equiv 1\bmod p$, if necessary).
In the scalar-valued case all these properties can be read off
from Sturm \cite{ST}, relying on earlier work of Shimura \cite{SH0}.
To show the validity of these properties for the vector-valued case
one may try to extend Sturm's method to this case. We do not know a good
reference for this.
A more elementary argument goes as follows :
In \cite{BKSP} we proved that theta series with pluriharmonic coefficients
generate the full space $M_n^k(\Gamma_0(p),\rho)$. This confirms
the $\nu_p$-boundedness, because it holds for such theta series.
On the other hand, the 
space of such pluriharmonic polynomials has a basis consisting of 
such polynomials with rational coefficients (if $\rho$ is 
${\mathbb Q}$-rational), therefore one obtains a 
solution of the basis problem by modular forms with rational coefficients.
For such theta series the action of $\omega_n$ can be written down explicitly
and has the requested property. Furthermore, in \cite{BFSP}
we gave explicit formulas for the trace of such theta series (as rational
linear combinations of other theta series, again with pluriharmonic
polynomials with rational coefficients).
\end{proof}
\begin{thm}
\label{th:3.4}
Let $p$ be an odd prime and $\rho$ a ${\mathbb Q}$-rational
representation. Then any $f\in M^k_n(\Gamma^n_0(p),{\rho})({\mathbb Q})$
is a p-adic modular form.
\end{thm}
\begin{proof} 
We apply Proposition \ref{prop:3.1} to $f$. From Proposition \ref{prop:3.3} a)
we see that the inverse image of $f$ under $U(p)$ has again Fourier coefficients
in ${\mathbb Q}^M$. The same is then true for $g$, $tr(g)$ and finally for $F$  
in Proposition \ref{prop:3.1}.
\end{proof}

\begin{rem}
\label{rem:3.5}
If we compare our result with Serre's in the degree one case,
our result is slightly weaker: It is possible that the application
of $\widetilde{U}_n(p)^{-1}$ introduces additional powers of $p$ 
in the denominator
(which weakens our congruences somewhat).
\end{rem}
\begin{rem}
\label{rem:3.6}
Theorem \ref{th:3.4} also holds for the case of nebentypus
$(\Gamma_0^n(p),\chi_p)$ with 
$\chi_p=\displaystyle \Big(\frac{(-1)^\frac{p-1}{2}p}{*}\Big)$. 
The proof is almost the same, however we should use
as ${\mathcal K}_{\frac{p-1}{2}}$ a theta 
series attached to a $p$-special lattice
of rank $p-1$ and determinant $p$.
\end{rem}

\begin{rem}
\label{rem:3.7}
In our proof of Proposition \ref{prop:3.3}
 we made essential use of theta series.
We mention that this  is not really necessary: 
One can (for arbitrary congruence
subgroups $\Gamma_0(N)$  and ${\mathbb Q}$ - rational representations $\rho$)
prove that the space $M^k_n(\Gamma_0(N),\rho)$ is spanned by
modular forms with rational Fourier coefficients with bounded denominators
(and cyclotomic Fourier coefficients with bounded denominators in the other
cusps)
by combining the line of reasoning of Garrett \cite{Ga} with
the use of suitable differential operators as in \cite{BKSP,I}.
\\
In a recent preprint Ichikawa \cite{ICHI2} treats these questions from the point of view of
arithmetic algebraic geometry. 
\end{rem}

\subsection{The case of higher level}
\label{sec:3.2}
In the previous argument, we showed that Siegel modular forms
for $\Gamma_0^n(p)$ are $p$-adic modular forms.
In this section, we shall generalize this fact to the case of
higher level. Namely, we show that $scalar$-$valued$ Siegel modular forms for
$\Gamma_0^n(p^m)$ are $p$-adic modular forms for any $m\in\mathbb{N}$.
We modify the arguments used by Serre (cf. Theorem 5.4, \cite{Ser2})
such that they apply in our situation.

Let $F_r\in M_n^r(\Gamma_0^n(p))(\mathbb{Z}_{(p)})$ denote any
modular form of weight $r$ and level $p$ such that
\[
F_r \equiv 1 \pmod{p}.
\]
The existence of such $F_r$ is assured by our previous considerations,
provided that $r$ is divisible by $p-1$. Furthermore we put for 
$i\geq 1$
\[
\mathcal{E}_i:=\prod_{j=0}^{i-1}F_{k(p-1)p^j}.
\]
This is a modular form of weight $k(p^i-1)$ and level $p$, i.e.,
$$\mathcal{E}_i\in
M_n^{k(p^i-1)}(\Gamma_0^n(p))(\mathbb{Z}_{(p)}).$$
\begin{lem}
\label{lem:3.8}  Assume that $m\geq 2$. For 
$G\in M_n^k(\Gamma_0^n(p^m))(\mathbb{Q})$ with $\nu_p(G)=0$, there is
$H\in M_n^{k\cdot p}(\Gamma_0^n(p^{m-1}))(\mathbb{Z}_{(p)})$ such that
\[
H \equiv G \pmod{p}.
\]
\end{lem}
\begin{proof} The modular form
\[
H:=(G^p)\mid U(p)
\]
has the properties requested above.
\end{proof}
\begin{prop}
\label{prop:3.9}
Assume that $m\geq 2$. Then for all
$G\in M_n^k(\Gamma_0^n(p^m))(\mathbb{Q})$ and all $i\geq 1$
there exists $H\in M_n^{k\cdot p^i}(\Gamma_0^n(p^{m-1}))(\mathbb{Q})$
such that
\[
G\cdot \mathcal{E}_i \equiv H \pmod{p^i}.
\]
\end{prop}
\begin{proof} We prove this fact by induction on $i$. The case
$i=1$ is the lemma above. For arbitrary $i$ we may assume that there
is $H_i\in M_n^{k\cdot p^i}(\Gamma_0^n(p^{m-1}))(\mathbb{Q})$ such
that
\[
H_i \equiv G\cdot\mathcal{E}_i \pmod{p^i}.
\]
Then we apply the lemma to the $p$-integral modular form
\[
\widetilde{G}:=\frac{1}{p^i}(G\cdot\mathcal{E}_i-H_i)
\in M_n^{k\cdot p^i}(\Gamma_0^n(p^m))(\mathbb{Z}_{(p)})
\]
and we obtain a modular form
$\widetilde{H}\in M_n^{k\cdot p^{i+1}}(\Gamma_0^n(p^{m-1}))(\mathbb{Q})$
such that 
$\widetilde{G}\cdot F_{kp^i(p-1)} \equiv \widetilde{H} \pmod{p}$ and
therefore 
\[
G\cdot\mathcal{E}_i\cdot F_{kp^i(p-1)} \equiv H_i\cdot F_{kp^i(p-1)}
+p^i\widetilde{H} \pmod{p^{i+1}}.
\]
If we put $H:=H_i\cdot F_{kp^i(p-1)}+p^i\widetilde{H}$
$\in M_n^{k\cdot p^{i+1}}(\Gamma_0^n(p^{m-1}))(\mathbb{Q})$, then we obtain
\[
G\cdot\mathcal{E}_{i+1} \equiv H \pmod{p^{i+1}}.
\]
This completes the induction.
\end{proof}
In the proposition above, the $p^{i-1}$-th power of $\mathcal{E}_i$
is then congruent one mod $p^i$ and we obtain
\begin{cor}
\label{cor:3.10} Let $G$, $m$, and $i$ as above. Then there
exists $H\in$\\
$M_n^l(\Gamma_0^n(p^{m-1}))$ $(\mathbb{Q})$ such that
\[
G \equiv H \pmod{p^i}.
\]
\end{cor}
As weight $l$ we may choose
\[
l=k+k(p^i-1)p^{i-1}.
\]
Now we can state the main theorem of this section.
\begin{thm}
\label{th:3.11}
 Every scalar-valued modular form
$G\in M_n^k(\Gamma_0^n(p^m))(\mathbb{Q})$ is a $p$-adic modular
form.
\end{thm}
\begin{proof}  We prove the statement by induction on $m$.
The case $m=1$ was handled before in the general situation
of vector-valued modular forms. 
For $m\geq 2$, we only to
need to know that $G$ is congruent modulo an arbitrary given 
power $p^i$ to a modular form of level $p^{m-1}$. This is guaranteed 
by the proposition above.
\end{proof}

\begin{rem} The result of the theorem above also holds for 
modular forms of quadratic nebentypus  $\chi_p$ (with obvious modifications
of proof).
\end{rem}

\begin{rem} As mentioned in the introduction, the vector-valued case is more 
difficult, because we cannot use the $p$-th power of 
a vector-valued modular
form. A good substitute for this is $p$-th symmetric power, this
however makes things more complicated, because we change 
the representation space by this procedure (we refer to \cite{BN3}
for details). In the next sections we avoid this problem by showing
only congruences involving modular forms of high level. These
results will be used in \cite{BN3} to construct vector-valued
$p$-adic Siegel modular forms (in an appropriate sense).
\end{rem}
  

\section{A large class of theta operators}
\label{sec4}
In the paper \cite{BN1}, we introduced a theta operator $\varTheta$ 
(cf. \cite{BN1}, p.428), and studied the arithmetic properties. 
For example, we showed that the algebra of Siegel modular forms 
mod $p$ is stable under the action of $\varTheta$ 
(\cite{BN1}, Corollary 3), and proved that $\varTheta(F)$ becomes a 
$p$-adic Siegel modular form if $F$ is an ordinary Siegel 
modular form for $\Gamma^n$ (\cite{BN1}, Theorem 5). However, 
the proof of Theorem 5 of that paper contains a defective argument. 
(In fact, the congruence relation of line 7 of p. 432 is not 
true in  general.) Here we give a complete proof of this 
theorem in a more general version including vector-valued generalizations
of the theta operator. Our proof is based on Rankin-Cohen brackets,
our Theorem \ref{th:3.11}
on arbitrary levels $p^m$ and the existence of modular forms 
congruent 1 mod $p$ . We point out that the method of proof is new even for
elliptic modular forms of level one !
\subsection{Rankin-Cohen brackets and general theta operators}
Rankin-Cohen operators for Siegel modular forms 
were investigated by
Ibukiyama \cite{I}, Eholzer/Ibukiyama \cite{E-I}
and many others. Beyond proving the existence of such operators \cite{I}, 
also explicit formulas
were considered. We try to avoid such explicit formulas as much as possible.
We fix a polynomial representation $\rho:GL_n({\mathbb C})\longrightarrow
GL(V_{\rho})$ and a weight $k$. We also assume as before, 
that $\rho$ comes up with a 
{\it fixed} matrix realization ($V_{\rho}= {\mathbb C}^M$). 
\\
    
We assume that we are given certain Rankin-Cohen bilinear operators

$$[,]_{k,l}:M_n^k(\Gamma)\times 
M^l_n(\Gamma)\longrightarrow 
M_n^{k+l}(\Gamma,\rho).$$
We consider
the case $l=(p-1)p^m$ with $m\geq 0$ varying.
We make two assumptions
\begin{description}
\item[(R-C 1)]
$[f,g]_{k,l}$ is a polynomial in the holomorphic 
derivatives of $f$ and $g$, more precisely, there
is a $V_{\rho}$-valued polynomial ${\mathcal P}={\mathcal P}_{k,l,\rho}(R_1,R_2)$,  
(with rational coefficients ) in the matrix 
variables $R_1, R_2\in Sym_n$, homogeneous of a certain degree 
$\lambda=\lambda(\rho)$ such that 
$$[f,g]_{k,l}= (2\pi\sqrt{-1})^{2-\lambda} {\mathcal P}(\partial_1,
\partial_2)(f(Z_1)\cdot g(Z_2))_{\mid Z_1=Z_2=Z}.$$

\item[(R-C 2)]
If we consider ${\mathcal P}$ as a polynomial in the variables $R_2$ alone
then we can decompose it into homogeneous components of degree $j$:
$${\mathcal P}=\sum_{j\geq 0} {\mathcal P}_j.$$ 
Then ${\mathcal P}_0$ should be independent of $l=(p-1)p^m$;
we define a $V_{\rho}$-valued differential operator $\varTheta_{k,\rho}$ by 
$$(2\pi\sqrt{-1})^{-\lambda}{\mathcal P}_0(\partial_1,\partial_2)
(f(Z_1)\cdot g(Z_2))_{\mid Z=Z_1=Z_2}=\varTheta_{k,\rho}(f)(Z)\cdot g(Z).$$ 
 Implicit in this assumption is a 
certain normalization of $[,]_{k,l}$. Note also that ${\varTheta}_{k,\rho}(f)$
has rational Fourier coefficients if $f$ has. This operator is a 
generalization of the well-known theta-operator from Serre \cite{Ser1}.
When $\rho$ is the one-dimensional representation $\text{det}^2$, 
$\varTheta=\varTheta_{k,\rho}$
is just one considered in the previous paper \cite{BN1}.
\end{description}
\subsection{Congruences for $\varTheta_{k,\rho}(f)$}
\label{subsec4.2}
Under the conditions above, we show
\begin{thm}
\label{th:4.1}  For any  $f\in M_n^k(\Gamma_0^n(p^r))(\mathbb{Q})$ and any $m\geq 0$  
there is $m'$ and
a modular form 
$F\in  M_n^{k+(p-1)p^{m-1}}(\Gamma_0^n(p^{m'}),\rho)(\mathbb{Q})$ such that
$${\varTheta}_{k,\rho}(f) \equiv F \pmod{p^m}.$$
\end{thm}
 \begin{proof}  We may assume that $\nu_p(f)=0$.
We choose a modular form  $F_{p-1}\in M_n^{p-1}(\Gamma^n_0(p))(\mathbb{Z})$
such that
\[
F_{p-1} \equiv 1 \pmod{p}.
\]
We choose arbitrary nonnegative integers $m,m'$ (to be specified later)
and we consider
the Rankin-Cohen bracket
$$
[f,F_{p-1}^{p^{m-1}}\mid V_{m'-1}]_{k,l}$$
where $(F_{p-1}^{p^{m-1}}\mid V_{m'-1})(Z):=F_{p-1}^{p^{m-1}}(p^{m'-1}Z)
\in M_n^{p^{m-1}(p-1)}(\Gamma_0^n(p^{m'}))$.\\
We investigate the ${\mathcal P}_j(\partial_1,\partial_2)
(f(Z_1)\cdot F_{p-1}^{p^{m-1}}\mid V_{m'-1}(Z_2))_{\mid Z=Z_1=Z_2}$ separately:
Clearly
$$(2\pi\sqrt{-1})^{-\lambda}{\mathcal P}_0(\partial_1,\partial_2)
(f(Z_1)\cdot F_{p-1}^{p^{m-1}}\mid V_{m'-1}(Z_2))_{\mid Z=Z_1=Z_2}=
{\varTheta}_{k,\rho}(f)(1+p^{m}G_0)$$
where
$G_0\in\mathbb{Z}_{(p)}[q_{ij}^{-1},q_{ij}]
[\![q_{11},\ldots,q_{nn}]\!].$

To study the contributions for $j\geq 1$, we write ${\mathcal P}_j$ as finite 
sum
of certain monomials 
(when considered as polynomials in the matrix variable $R_2$):

$${\mathcal P}_j=\sum_{\alpha} P_{j,\alpha}\cdot Q_{j,\alpha}$$ 
where the $P_{j,\alpha}$ denote  polynomials in the variable $R_1$
and $Q_{j,\alpha}$ denotes a normalized monomial of degree $j$ in the 
variable $R_2$. 

Returning to the calculation, we write as
\[
f(Z)=\sum_T a(T)\boldsymbol{q}^T,\quad (F_{p-1}^{p^{m-1}}\mid V_{m'-1})(Z)=1+
\sum_{T\ne 0}p^{m}b'(T)\boldsymbol{q}^{p^{m'-1}T}.
\]
Then
\begin{align*}
&(2\pi\sqrt{-1})^{-\lambda}P_{j,\alpha}(\partial_1)\cdot Q_{j,\alpha}(\partial_2)
(f(Z_1)
\cdot F^{p^{m-1}}_{p-1}\mid V_{m'-1}(Z_2))_{\mid Z=Z_1=Z_2}\\
&=
\left(\sum_T a(T)P_{j,\alpha}(T)\boldsymbol{q}^T\right)
\left(\sum_{T\not=0} p^{m}b'(T) p^{j(m'-1)} Q_{j,\alpha}(T)
\boldsymbol{q}^{p^{m'-1}T}\right)\\
&=p^{m+j(m'-1)}\cdot G_{j,\alpha}
\end{align*}
where the possible $p$-denominators in the Fourier coefficients 
of $G_{j,\alpha}$
depend only on the polynomial $P_{j,\alpha}$; note that this polynomial
may depend on $m$, but not on $m'$.
Now we choose $m'$ sufficiently large to guarantee that 
$$p^{j(m'-1)}G_{j,\alpha}\in 
\mathbb{Z}_{(p)}[q_{ij}^{-1},q_{ij}]
[\![q_{11},\ldots,q_{nn}]\!]^{V_{\rho}}.$$
The upper index $V_{\rho}$ indicates that we deal with polynomials
with values in the vector space $V_{\rho}$.\\
We have achieved in this way that ${\varTheta}_{k,\rho}(f)$ is
congruent mod $p^{m}$ to a vector-valued modular form of level
$\Gamma_0^n(p^{m'})$. 
\end{proof}
\noindent
\begin{rem}
\label{rem:4.2}
The same is true in the case of real nebentypus $\chi_p$
(by a suitable modification of the proof above).
\end{rem}

\begin{rem}
\label{rem:4.4}
With a little bit more efforts 
it is also possible to formulate the Theorem above 
for the case where $f$ is already vector-valued.
\end{rem}
\begin{rem}
\label{rem:4.5}
The simplest possible example for the Rankin-Cohen bracket
(for $n=1$, $\rho=\det^2$) shows, that we cannot avoid using the
operator $V_{m'}$ in our proof (this is also a good example to illustrate
our normalization):
$$[f,g]_{k,(p-1)p^m}=(2\pi\sqrt{-1})^{-1} 
\left(f'\cdot g - \frac{k}{(p-1)p^m}f\cdot g'\right).$$
To compensate the denominator $p^m$, we have to 
use the level raising operator $V_m$.
\end{rem}
The Theorem 4.1 is not completely satisfying because of the possibly very high
level of the modular form $F$. If we impose the additional condition
\begin{description}
\item[(R-C 3)] The coefficients of the polynomial ${\mathcal P}$ depend
continously (in the p-adic sense) on the weight $l$.
\end{description}
We can improve Theorem 4.1 significantly by a slight modification of the
proof:
\begin{thm}
\label{th:4.1'}  Assume that the Rankin-Cohen operators
satisfy the conditions (R-C 1)-(R-C 3). Then, for 
any  $f\in M_n^k(\Gamma_0^n(p^r))(\mathbb{Q})$ and 
any $m\geq 0$  
there is a weight $k'$ and 
a modular form 
$G\in  M_n^{k'}(\Gamma_0^n(p^{r}),\rho)(\mathbb{Q})$ such that
$${\varTheta}_{k,\rho}(f) \equiv G \pmod{p^m}.$$ In particular,
${\varTheta}_{k,\rho}(f)$ defines a $p$-adic modular form.
\end{thm}
\begin{proof}
We proceed as follows:\\
We start by the same procedure as before, investigating
$$[f,F^{p^{m-1}}_{p-1}\mid V_{m'-1}]_{k,l},$$
assuring the congruence

$$\varTheta_{k,\rho}(f)\equiv F\bmod p^m$$
as in theorem 4.1.

Now we observe, that $F^{p^{m-1}}_{p-1}\mid V_{m'-1}$ is
congruent to a modular form $H$ of level one, weight $l'$, more precisely,
there is a weight $l'=(p-1)p^{m-1}+\alpha(p-1)p^{N'}$
such that 
$$F^{p^{m-1}}_{p-1}\mid V_{m'-1}\equiv H\bmod p^N$$

Now we consider the bracket

$$[f,H]_{k,l'}$$
instead of the bracket above.
If we choose $N,\,N'$ sufficiently large, we do not get new $p$-denominators
 (because of the continuity condition)
and we get that 
$$\varTheta_{k,\rho}(f)$$
is congruent modulo  a power of $p$ to a 
modular form of level $p^r$.
\\
To see that $\varTheta_{k,\rho}$ is $p$-adic, we have to modify the procedure above still further:
We substitute $f$ modulo an arbitrary power of $p$ by a form $\tilde{f}$ of level $\Gamma_0(p)$.
Then the $G$ as above is also of level $\Gamma_0(p)$. We may then apply Theorem \ref{th:3.4}.

\end{proof}

\subsection{A remark on ${\mathcal P}_0$}
\label{subsec4.3}
Here we show that ${\mathcal P}_0$ is uniquely determined 
(up to a scalar factor) by
a certain invariance property, if $\rho$ is irreducible.\\
The bilinear differential operator $[,]_{k,l}$ defines a polynomial function
$f:Sym_n({\mathbb C})\longrightarrow V_{\rho}$ by
$$\left\{[e_S,e_T]_{k,l}\right\}_{\mid T=0} = f(S)e_S,$$
where $S$ and $T$ are symmetric complex matrices of size $n$ and $e_S$
denotes the function on ${\mathbb H}_n$ defined by 
$$Z\longmapsto \text{exp}(2\pi\sqrt{-1} \text{tr}(S\cdot Z)).$$
The invariance properties of $[,]_{k,l}$ give
$$f(A^tSA)= \rho(A)f(S),$$  
first for all $A\in GL_n({\mathbb R})$ and hence also for 
$A\in GL_n({\mathbb C})$ and then also for $A\in M_n({\mathbb C})$.
Every symmetric complex matrix $S$ can be written as $S=A^t\cdot 1_n\cdot A$
with $A\in M_n({\mathbb C})$.
Therefore $f$ is determined by ${\mathbf v}:= f(1_n)\in V_{\rho}$.
This vector ${\mathbf v}$ is 
$O(n,{\mathbb C})$ invariant. The space of such invariants
is at most one-dimensional, if $\rho$ is irreducible. This follows from
branching rules (see e.g.
\cite{K-T}) or from $GL(n,{\mathbb C}),O(n,{\mathbb C})$ being a Gelfand pair
\cite{AG}.
\begin{rem}
\label{rem:4.6}
 The problem of (non-) vanishing of  ${\mathcal P}_0$ deserves further 
investigation.
\end{rem}
\subsection{Theta operators and Maa{\ss} differential operators}
\label{subsec4.4}
Let us now give some ``basic'' examples of such operators:
Note that the existence of certain Rankin-Cohen brackets is not 
sufficient, we must know the nature of the ``constant term'' 
${\mathcal P}_0$. It is desirable to show quite generally that
${\mathcal P}_0$ is different from zero (say, if $l=(p-1)p^m$
is large). We consider here only a simple type:
Our exposition follows \cite{FR} and \cite{WE};  we also use some elements
of Shimura's theory of nearly holomorphic functions \cite{SH}.

We fix a decomposition $n=r+s$.
For a complex matrix $A$ of size $n$ we 
denote by $A^{[r]}$ the matrix of size $n\choose r$ whose entries
are the $r$-minors of $A$; here we fix an order among the
subsets of $\{1,\dots, n\}$ with $r$ elements.

We get an irreducible representation $\rho^{[r]}$ of $GL_n(\mathbb{C})$
on $Sym_{n\choose r}({\mathbb C})$ by 
$$(A, X)\longmapsto \rho^{[r]}(A)(X):= A^{[r]}\cdot X\cdot (A^{[r]})^t.$$
We denote by $\partial :=(\partial_{ij})$ the $n\times n$-matrix defined by
partial derivatives on ${\mathbb H}_n$:
$$\partial_{ij}=\left\{\begin{array}{ccc}
\frac{\partial}{\partial z_{ii}} & \mbox{if}& i=j\\
\\
\frac{1}{2}\frac{\partial}{\partial z_{ij}} & \mbox{if}& i\not=j
\end{array}\right. .$$
We consider the differential operator 
$\partial^{[r]}$ 
which maps $C^{\infty}$ functions on ${\mathbb H}_n$ to 
$Sym_{n\choose r}({\mathbb C})$-valued functions.
The transformation properties of this operator are well-known
(\cite{FR}, p.214)

$$\partial^{[r]}(h\mid_{\frac{r-1}{2}}M)= \left(\partial^{[r]}h\right)
\mid_{\det^{\frac{r-1}{2}}\otimes \rho^{[r]}}M\qquad (M\in Sp_n({\mathbb R})).$$

For arbitrary weight $k$ we consider now the Maa{\ss}-type differential operator

$$(\partial^{[r]}_k) h:= 
\det(Y)^{-k+\frac{r-1}{2}}\cdot\partial^{[r]}h\cdot \det(Y)^{k-\frac{r-1}{2}}.
$$
This operator changes the automorphy factor from $\det^k$ to 
$\det^k\otimes \rho^{[r]}$.

For arbitrary holomorphic functions $f,g$ on ${\mathbb H}_n$
we consider 
$$(\partial_k^{[r]}f)\cdot g.$$
This is a nearly holomorphic function in the sense of Shimura \cite{SH}.
We think of $f$ and $g$ to carry a $\mid_k$ and $\mid_l$ action of 
$Sp_n({\mathbb R})$ (respectively). Then a structure theorem of 
Shimura \cite{SH} on
such nearly holomorphic functions (with $Sp_n({\mathbb R})$ acting 
by $\mid_{\text{det}^{k+l}\otimes \rho^{[r]}}$) says that (provided that $k+l$ is 
large enough)

$$ (\partial_k^{[r]}f)\cdot g = B_{k,l,\rho}(f,g) +\Delta$$
where $B_{k,l,\rho}(f,g)$ is a holomorphic function and $\Delta$
is a finite sum of images of certain holomorphic functions under 
differential operators of Maa{\ss}-Shimura type.
Analytically $B_{k,l,\rho}(f,g)$ is the ``holomorphic projection''
of  $(\partial_k^{[r]}f)\cdot g$ (at least if $f,g$ are actually modular forms
satisfying certain growth conditions). 
Actually, the decomposition above is of purely algebraic nature and 
an inspection of Shimura's proof shows that the expression
for $B_{k,l\rho}(f,g)$ is a bilinear form in the derivatives of $f$ and $g$
(only depending on $k,l,\rho $), i.e. it is a bilinear differential operator
of Rankin-Cohen type. The coefficients of the derivatives of $f$ and $g$
are rational functions of $k$ and $l$ over  ${\mathbb Q}$. 
The representation $\rho^{[r]}$ is irreducible and using the reasoning of $4.3$ 
we see that the part of $B_{k,l,\rho}$, which is free of derivatives of $g$ 
must be a multiple of $(\partial^{[r]}f)\cdot g $.\\
Now we consider the coefficient $q_{\rho}(k,l)$ of $(\partial^{[r]}f)\cdot g$
in $B_{k,l,\rho}(f,g)$. We observe that this rational function is not identically
zero because clearly
$$(\partial_{\frac{r-1}{2}}^{[r]}f)\cdot g= B_{\frac{r-1}{2},l,\rho}(f,g)  \qquad 
(l>\!\!>0).$$
This implies: There is a finite set $M$ such that
for all $k\in {\mathbb Z}\setminus M$ we have $q_{\rho}(k,l)\not= 0$ for $l>\!\!>0$.
\\
We can rephrase this as follows:
\begin{prop}
\label{prop:4.7}
 For $k\notin M$ there is for $l>\!\!>0$ a 
Rankin-Cohen differential operator
$$[f,g]_{k,l,\rho^{[r]}}= (\partial^{[r]}(f))\cdot g +
\sum_{j\geq 1} {\mathcal P}_j(\partial_1,\partial_2)
(f(Z_1)g(Z_2))_{\mid Z=Z_1=Z_2}$$
satisfying the conditions (R-C 1) - (R-C 3).   
\end{prop}

For $1\leq r\leq n$, we put
\[
\varTheta_{k,\rho^{[r]}}:=(2\pi\sqrt{-1})^{-r}\partial^{[r]}.
\]
\begin{thm}
\label{th:4.8}
 Assume that $1\leq r\leq n$. For any
$f\in M_n^k(\Gamma_0^n(p^m))(\mathbb{Q})$ 
the formal power series $\varTheta_{k,\rho^{[r]}}(f)$
is a p-adic modular form.
\end{thm}
\begin{proof}
Using proposition \ref{prop:4.7} we may apply Theorem \ref{th:4.1'}.
We observe that by multiplying $f$ by a power of $F_{p-1}$ we may avoid
the finite set $M$ of  prop. \ref{prop:4.7}; furthermore, in the same way, the assumption
$l>>0$ can be achieved.

\end{proof}

The scalar-valued case ($r=n$) of the theorem above repairs
a gap in the proof of Theorem 5 in \cite{BN1} concerning the ordinary $\varTheta$-operator:
\begin{cor} For any $f\in M^k_n(\Gamma_0(p^m))({\mathbb Q})$
the derivative $\varTheta(f)$ is a $p$-adic modular form.
\end{cor}
\subsection{A combinatorial approach (Note added in proof)}

Recently we gave a purely combinatorial approach to the 
differential operators $B_{k,l,\rho}(f,g)$ in the context of
hermitian modular forms \cite{BD}. 
The same approach (avoiding the nonholomorphic differential operators
of subsection 4.4.) also works for the Siegel case as follows: \\
For $0\leq \alpha\leq r\leq n$ we define a polynomial in the matrix 
variables $R,S\in Sym_n({\mathbb C})$ with values in 
$Sym_{n\choose r}({\mathbb C})$ by
$$(R+\lambda S)^{[r]}=\sum_{\alpha=0}^r P_{\alpha,r}(R,S) \lambda^{\alpha}$$
In the notation of \cite[III.6]{FR} we have 
$$
P_{\alpha,r}(R,S)= \left(\begin{array}{c} r\\ \alpha\end{array}\right)
R^{[\alpha]}\sqcap S^{[r-\alpha]}.$$
We put
$$C_h(s)=s\cdot (s+\frac{1}{2})\dots \cdot 
(s+\frac{h-1}{2})\qquad (0\leq h\leq n)$$
Then one can show along the same lines as in \cite{BD} that
$$D(f,g):= \sum_{\alpha=0}^r (-1)^{\alpha}  
C_{\alpha}(l-\frac{r-1}{2}) \cdot C_{r-\alpha}(k-\frac{r-1}{2}) 
{r \choose \alpha}
\partial^{[\alpha]}(f)\sqcap 
\partial^{[r-\alpha]}(g)
$$
is an explicit realization of the Rankin-Cohen bracket $[f,g]_{k,l,\rho^{[r]}}$.

\subsection*{Acknowledgements}
Crucial work on this paper was done during our stay
at the Mathematisches Forschungsinstitut Oberwolfach under the programme
``Research in Pairs'' ; we continued our work during research visits at
Kinki University and Universit\"at Mannheim (respectively); a final revision
was done, when the first author held a guest professorship
at the University of Tokyo. We thank these
institutions for support. We also thank Dr.Kikuta for  
pointing out some gaps in our presentation and Professor T.Ichikawa for
discussions about $p$-adic modular forms.

\end{document}